\documentclass[11pt]{article}
\usepackage{amsfonts}
\usepackage{amsmath}
\usepackage{amssymb}
\usepackage{amsthm}
\usepackage{setspace}
\usepackage{fullpage}
\usepackage{hyperref}
\usepackage{url}

\newtheorem{thm}{Theorem}[section]
\newtheorem{prop}[thm]{Proposition}
\newtheorem{lem}[thm]{Lemma}
\newtheorem{conj}[thm]{Conjecture}
\newtheorem{ex}[thm]{Example}

\newtheorem{case}{Case}[thm]
\newtheorem{cor}[thm]{Corollary}
\newtheorem*{rem}{Remark}
\newtheorem*{Def}{Definition}
\newtheorem*{thm:monotone}{Theorem \ref{thm:monotone}}
\newtheorem*{thm:product}{Theorem \ref{thm:product}}
\newtheorem*{thm:digits}{Theorem \ref{thm:digits}}
\newtheorem*{thm:shift}{Theorem \ref{thm:shift}}

\newcommand{\lb}[0]{\lambda}
\newcommand{\Gm}[0]{\Gamma}
\newcommand{\Gmt}[0]{\Gamma^\ast}
\newcommand{\at}[0]{a^\ast}
\newcommand{\Xt}[0]{\Lambda^\ast}
\newcommand{\Lb}[0]{\Lambda}
\newcommand{\s}[0]{\sigma}
\newcommand{\om}[0]{\omega}
\newcommand{\mA}[0]{\mathcal{A}}

\title{On the classification of Stanley sequences}
\author{David Rolnick\thanks{Massachusetts Institute of Technology, Cambridge MA. Email: \url{drolnick@math.mit.edu}}}

\begin{document}

\maketitle

\begin{abstract}
An integer sequence is said to be \emph{3-free} if no three elements form an arithmetic progression.  Following the greedy algorithm, the \emph{Stanley sequence} $S(a_0,a_1,\ldots,a_k)$ is defined to be the 3-free sequence $\{a_n\}$ having initial terms $a_0,a_1,\ldots,a_k$ and with each subsequent term $a_n>a_{n-1}$ chosen minimally such that the 3-free condition is not violated.  Odlyzko and Stanley conjectured that Stanley sequences divide into two classes based on asymptotic growth patterns, with one class of highly structured sequences satisfying $a_n\approx \Theta(n^{\log_2 3})$ and another class of seemingly chaotic sequences obeying $a_n=\Theta(n^2/\log n)$.  We propose a rigorous definition of \emph{regularity} in Stanley sequences based on local structure rather than asymptotic behavior and show that our definition implies the corresponding asymptotic property proposed by Odlyzko and Stanley.  We then construct many classes of regular Stanley sequences, which include as special cases all such sequences previously identified.   We show how two regular sequences may be combined into another regular sequence, and how parts of a Stanley sequence may be translated while preserving regularity.  Finally, we demonstrate that certain Stanley sequences possess proper subsets that are also Stanley sequences, a situation that appears previously to have been assumed impossible.
\end{abstract}

\section{Introduction}
A set of non-negative integers is \emph{3-free} if no three elements form an arithmetic progression.  Given a 3-free set $A$ with elements $a_0<a_1<\cdots<a_k$, we define the \emph{Stanley sequence} $S(A)=\{a_n\}$ according to the greedy algorithm, as follows: Assuming $a_n$ has been defined, let $a_{n+1}$ be the smallest integer greater than $a_n$ such that $\{a_0,\ldots,a_{n+1}\}$ is 3-free.  For convenience, we shall often write $S(a_0,a_2,\ldots,a_k)$ for $S(\{a_0,a_1,\ldots,a_k\})$.

The simplest Stanley sequence is $S(0)=0,1,3,4,9,10,12,13,27,\ldots$, the elements of which are exactly those integers with no 2's in their ternary representation.  Odlyzko and Stanley \cite{stanley} offered similar closed-form descriptions of the sequences $S(0,3^n)$ and $S(0,2\cdot 3^n)$, for $n$ any non-negative integer.  Their work also suggested an overarching dichotomy among Stanley sequences, in which the more ``structured" sequences (such as $S(0)$) followed one asymptotic growth pattern, while more ``chaotic" sequences followed another.

\begin{conj}[based on work by Odlyzko and Stanley \cite{stanley}] Let $S(A)=\{a_n\}$ be a Stanley sequence.  Then, for all $n$ large enough, one of the following two patterns of growth is satisfied:
\begin{enumerate}
\item $\frac{c}{2}\cdot n^{\log_2 3}\le a_n\le c\cdot n^{\log_2 3}$, or
\item $a_n\approx c'n^2/\log n$.
\end{enumerate}
\label{stanconj}
\end{conj}

\begin{rem} The original paper \cite{stanley} considered the first type of growth in the case of $c=1$ only.  However, if $c$ is so restricted, the conjecture is certainly false, with $S(0,1,7)$ being one counterexample, requiring $c=10/9$.  (This assertion is simple to prove with machinery we will present in \S 2 and \S 3.)
\end{rem}

We will call these two patterns of growth \emph{Type 1} and \emph{Type 2}.   The closed-form descriptions given in [1] of $S(0,3^n)$ and $S(0,2\cdot 3^n)$ demonstrate that these sequences do indeed follow Type 1 growth, but they are by no means the only such sequences.  The justification given in \cite{stanley} for conjecturing Type 2 growth is a non-constructive probabilistic method that suggests, but does not prove, that a ``random" Stanley sequence should follow Type 2 growth.  However, no particular sequence has yet been shown to be Type 2, though Lindhurst \cite{lindhurst} has provided extensive data to support the notion that $S(0,4)$ follows this type of growth.

Erd\H{o}s et al.~\cite{erdos} posed several problems similar to Conjecture \ref{stanconj}, regarding the density of Stanley sequences.  Moy \cite{moy} recently solved one of these problems by showing that all Stanley sequences $\{a_n\}$ satisfy the asymptotic bound$$a_n\le x^2/(2+\epsilon).$$

In this paper, we approach the conjectured dichotomy among Stanley sequences from the perspective of local structure, rather than asymptotic behavior.  We begin, in Section \ref{sec:regular}, by identifying two important types of Stanley sequence.

\begin{Def} We say that a Stanley sequence $S(A)=\{a_n\}$ is \emph{independent} if there exists a constant $\lb$, called the \emph{character}, such that, for all sufficiently large $k$, the equations \begin{eqnarray*}a_{2^k+i}&=&a_{2^k}+a_i\\ a_{2^k}&=&2a_{2^k-1}-\lb+1\end{eqnarray*}hold whenever $0\le i< 2^k$.  We say that an integer $k_0$ is \emph{adequate} if (i) these equations are satisfied for all $k\ge k_0$ and (ii) $k_0$ is large enough that $a_{2^{k_0}}$ is not a necessary element of the set $A$.
\end{Def}

\begin{Def}
We say that a Stanley sequence $S(A)=\{a_n\}$ is \emph{regular} if there exist constants $\lb, \sigma$ and an independent Stanley sequence $\{a'_n\}$, having character $\lb$, such that, for large enough $k$ and $0\le i<2^k$,$$a_{2^k-\sigma+i}=a_{2^k-\sigma}+a'_i\hskip .5in a_{2^k-\sigma}=2a_{2^k-\sigma-1}-\lb+1.$$We refer to the sequence $\{a'_n\}$ as the \emph{core} of $S(A)$ and the constant $\sigma$ as the \emph{shift index}.
\end{Def}

Regularity is a strictly weaker condition than independence.  We show that all regular Stanley sequences follow Type 1 growth, and we conjecture the converse.

In Section \ref{sec:indep}, we consider methods for constructing independent Stanley sequences.  We begin by describing a class of independent Stanley sequences that includes as a special case the sequences $S(0,3^n)$ and $S(0,2\cdot 3^n)$ detailed in \cite{stanley}.

\begin{thm}[see Example \ref{ex:monotone}]
\label{thm:monotone}
Let $k$ be a positive integer and $\mathcal{A}$ be a monotone decreasing family of subsets of $\{0,1,\ldots,k-1\}$.  Let $$A=\{3^{a_1}+3^{a_2}+\cdots+3^{a_n}\mid \{a_1,a_2,\ldots,a_n\}\in \mathcal{A}\}.$$Then, $S(A\cup \{3^k\})$ and $S(A\cup \{2\cdot 3^k\})$ are independent.  (In Section \ref{sec:indep}, we give closed-form descriptions of these sequences.)
\end{thm}

We next describe an operation that combines a regular with an independent sequence to yield a regular sequence.

\begin{thm}[see Example \ref{ex:product}]
\label{thm:product}
Let $S(A)=\{a_n\}$ be independent and $S(B)=\{b_n\}$ be regular.  Let $k$ be adequate with respect to $S(A)$.  Let $A^{\ast}=\{a_0,a_1,\ldots,a_{2^k-1}\}$ and define $$A\otimes_k B=\{a_{2^k}b+a\mid  a\in A^{\ast},b\in B\}.$$Then, $S(A\otimes_k B)$ is a regular Stanley sequence, independent if and only if $B$ is, having description $$S(A\otimes_k B)=\{a_{2^k}b+a\mid  a\in A^{\ast},b\in S(B)\},$$with character $\lb(A\otimes_k B)=a_{2^k}\cdot \lb(B)+\lb(A)$ and shift index $\sigma(A\otimes_k B)=2^k\cdot \sigma(B)$.
\end{thm}

Using this operation, we describe another class of independent Stanley sequences.

\begin{thm}[see Example \ref{ex:digits}]
\label{thm:digits}
Let $k$ be a positive integer. Let $T_1,T_2$ be disjoint subsets of $\{0,1,\ldots,k\}$ such that no $t\in T_1$ satisfies $t-1\in T_2$.  Let $$A=\left\{\left(3^{a_1}+3^{a_2}+\cdots+3^{a_m}\right)+2\left(3^{b_1}+3^{b_2}+\cdots+3^{b_n}\right)\mid \{a_1,\ldots,a_m\}\subseteq T_1,\{b_1,\ldots,b_n\}\subseteq T_2\right\}.$$Then, $S(A)$ is an independent Stanley sequence.
\end{thm}

In Section \ref{sec:dep}, we turn to dependent Stanley sequences.  Given an independent sequence $S(A)$, one may obtain a dependent sequence by translating a portion of $S(A)$ and recomputing subsequent elements.  The resulting sequence has $S(A)$ as its core.

Given the regular sequence $S(A)=\{a_n\}$ and the nonnegative integers $k$ and $c$, let $S_k(c,A)$ be the Stanley sequence generated by the set $$A_k(c,A)=\left\{a_i\mid 0\le i<2^k-\sigma(A)\right\} \cup \left\{c+a_i\mid 2^k-\sigma(A)\le i<2^{k+1}-\sigma(A)\right\},$$assuming that this set is 3-free (if it is not, then $S_k(c,A)$ is not defined).

\begin{thm}[see Example \ref{ex:shift}]
\label{thm:shift}
Let $S(A)=\{a_n\}$ be an independent sequence with character $\lb$.  Let $\ell$ be the minimum adequate integer for $S(A)$, and pick $k\ge \ell$. Let $c$ be such that$$\lb\le c\le a_{2^k-2^\ell}-\lb.$$Then, $S_k(c,A)$ is defined and is a regular Stanley sequence, with core $S(A)$.
\end{thm}

We conclude by noting that, for some regular sequences, an element may be removed without changing the Stanley sequence property of the other elements.  (See Example \ref{ex:deletion}.)  This disproves the tacit assumption in Erd\H{o}s et al.~\cite{erdos} and Moy \cite{moy} that Stanley sequences are \emph{maximal} 3-free sets (3-free sets that cannot be strictly enlarged).

\section{Regular Stanley sequences}
\label{sec:regular}

We begin by introducing some useful terminology and notation.  If $S(A)$ is a Stanley sequence, we say that $A$ is a \emph{nucleating set} of $S(A)$.  Note that a given Stanley sequence has infinitely many nucleating sets, corresponding to all sufficiently large prefixes of the sequence.  We define the \emph{minimal} nucleating set of a Stanley sequence as the one which is of minimal cardinality.

We use the shorthand $A+n$ to denote the set $\{a+n\mid a\in A\}$, for any set $A$ and integer $n$.  It is easy to see that if $S(A)$ is a Stanley sequence and $n$ is a non-negative integer, then $$S(A+n)=S(A)+n.$$In other words, translating the nucleating set translates the entire sequence.  It is sufficient, therefore, when investigating Stanley sequences to consider only those which begin at 0.  We say that such Stanley sequences are in \emph{root position} and for the remainder of this paper will assume that all Stanley sequences under consideration are in root position.

We say an integer $x$ is \emph{covered} by a set $S$ of integers if there exist $s,t\in S$ such that $s<t$ and $2t-s=x$.  Then, the Stanley sequence $S(a_0,a_1,\ldots,a_k)$ is the unique increasing sequence $S=\{a_n\}$ where each integer $x>a_k$ is covered by $S$ if and only if it is not in $S$.   Given a Stanley sequence $S(A)$, we define the \emph{omitted set} $O(A)$ to be the set of nonnegative integers that are neither in $S(A)$ nor are covered by $S(A)$.  We let $\om(A)$ denote the largest element of $O(A)$. It is immediate that $\om(A)$ is less than the largest element of $A$.

We say that an integer $x$ is \emph{jointly covered} by sets $S$ and $T$ if there exist $s\in S$, $t\in T$ such that $s<t$ and $2t-s=x$.  Thus, an integer jointly covered by $S$ and $S$ is covered by $S$.  We say that a set $X$ is covered (or jointly covered) by a set $S$ (or pair of sets $S$ and $T$) if every element of $X$ is so covered.

The following lemma is trivial to prove but will be extremely useful hereafter.

\begin{lem}[Cover-shift Lemma] If $x$ is jointly covered by $S$ and $T$, and if $n_1\le n_2$ are integers, then $x+(2n_2-n_1)$ is jointly covered by $S+n_1$ and $T+n_2$.
\label{covershift}
\end{lem}

We are now ready to introduce the classes of Stanley sequence that are the subject of this paper.

\subsection{Independent sequences}

\begin{Def} We say that a Stanley sequence $S(A)=\{a_n\}$ is \emph{independent} if there exists a constant $\lb$ such that, for all sufficiently large $k$, the equations \begin{eqnarray}a_{2^k+i}&=&a_{2^k}+a_i\label{an}\\ a_{2^k}&=&2a_{2^k-1}-\lb+1\label{another}\end{eqnarray}hold whenever $0\le i< 2^k$.
\end{Def}

It is evident from this definition that the constant $\lb$ is unique.  We may therefore refer to it as the \emph{character} $\lb(A)$ of the independent sequence $S(A)$.

For each $k$, we will refer to the set $\{a_i\mid 2^k\le i<2^{k+1}\}$ as the \emph{$k$th block $\Gm_k$}.

\begin{ex}
The Stanley sequence $S(0,2,5)$ is independent with character $\lb=4$.

$$S(0,2,5)=0,2,\underbrace{\fbox{$5,6,$}}_{\Gm_1}\underbrace{\fbox{$9,11,14,15,$}}_{\Gm_2}\underbrace{\fbox{$27,29,32,33, 36, 38, 41, 42,$}}_{\Gm_3}\ldots$$

Note that, in moving from the last element of one block to the first element of the next block, the value of the sequence doubles and subtracts $(\lb-1)$.  The next block is then a shift of the preceding terms.  Thus, for instance,$$2\cdot a_3-\lb+1=2\cdot 6 -4+1=9=a_4$$and $\{9,11,14,15\}=\{0,2,5,6\}+9$.  Likewise, $2\cdot 15-4+1=27$, and $\{27,29,32,33, 36, 38, 41, 42\}=\{0,2,5,6,9,11,14,15\}+27$.
\end{ex}

The following proposition shows that the criterion ``for all sufficiently large $k$" in the above definition can be replaced by ``for a single sufficiently large $k$."

\begin{prop} Let $S(A)=\{a_n\}$ be a Stanley sequence with $\om=\om(A)$, and suppose integers $\lb$ and $k$ are such that $a_{2^k-1}\ge \lb+\om$ and that equations (\ref{an}) and (\ref{another}) hold whenever $0\le i<2^k$.  Then, $S(A)$ is independent with character $\lb$.
\end{prop}
\begin{proof} It suffices to show that (\ref{an}) and (\ref{another}) must hold for all $k'>k$, and hence to show that they hold if $k$ is replaced with $k+1$.  Let $\Lb=\{a_i\mid 0\le i<2^k\}$ and $\Gm=\{a_i\mid 2^k\le i<2^{k+1}\}$, so that $\Gm=\Lb+a_{2^k}$ (see Figure \ref{fig}).  Let $B$ be the set of integers in the interval $[0,a_{2^k-1}]$ that are covered by $\Lb$.

Our strategy will be to describe the integers covered by $\Lb\cup \Gm$ by breaking up this set into (i) the integers covered by $\Lb$ alone, (ii) the integers covered by $\Gm$ alone, (iii) the integers jointly covered by $\Lb$ and $\Gm$.  We additionally break up the set in (iii) into the sets $\{2y-x\mid x\in \Lb, y\in \Gm, y>x+a_{2^k}\}$, $\{2y-x\mid x\in \Lb, y\in \Gm, y=x+a_{2^k}\}$, and $\{2y-x\mid x\in \Lb, y\in \Gm, y<x+a_{2^k}\}$.

We begin by observing that $\Lb$ covers the following integers:
\begin{itemize}
\item $B$, by definition.
\item The open interval $(a_{2^k-1},a_{2^k})$, since these integers are not in $S(A)$ and hence must be covered by $S(A)$.
\item The set $O(A)+a_{2^k}$.  To see this, observe that each element $s\in O(A)+a_{2^k}$ must be covered by $S(A)$, and yet cannot be covered by $\Lb+a_{2^k}=B$ by the definition of $O(A)$.  Hence, there must be $x\in \Lb$ and $w\in S(A)$ such that $2w-x=s$.  Since $x\le a_{2^k-1}$ and $s\le \omega+a_{2^k}$, we conclude that $$2w\le a_{2^k-1}+\omega+a_{2^k}\le 2a_{2^k-1}-\lb+1+a_{2^k}$$ (where the second inequality follows from $a_{2^k-1}\ge \lb+\om$).  The right side of this equals $2a_{2^k}$, implying that $w<a_{2^k}$ and hence that $w\in \Lb$.  We conclude that $s$ and hence $O(A)+a_{2^k}$ is covered by $\Lb$.
\end{itemize}

\begin{figure}
\begin{eqnarray*}
S(0,2,5)&=&0,2,5,6,9,11,14,15,27,29,32,33, 36, 38, 41, 42, 81,\ldots\\
\lb&=&4\\
O(A)&=&\{1,3,4\}\\
\Lb&=& \{a_0,a_1,a_2,a_3\}=\{0,2,5,6\}\\
\Gm&=&\Lb+a_4= \{9,11,14,15\}\\
\end{eqnarray*}
\caption{The independent Stanley sequence $\{0,2,5\}$, with character $\lb=4$.}
\label{fig}
\end{figure}

It is easy to see that the union $$B\cup (a_{2^k-1},a_{2^k})\cup \left(O(A)+a_{2^k}\right)$$ of these three sets in fact constitutes \emph{exactly} the integers covered by $\Lb$. Hence, by the Cover-shift Lemma (Lemma \ref{covershift})
\begin{itemize}
\item The set $\Gm=\Lb+a_{2^k}$ must cover exactly the union $$\left(B+a_{2^k}\right)\cup (a_{2^{k+1}-1},2a_{2^k})\cup \left(O(A)+2a_{2^k}\right).$$
\item The set $\{2y-x\mid x\in \Lb, y\in \Gm, y>x+a_{2^k}\}$ equals the union $$\left(B+2a_{2^k}\right)\cup (a_{2^k-1}+2a_{2^k},3a_{2^k})\cup \left(O(A)+3a_{2^k}\right).$$
\end{itemize}

We now note that $$\{2y-x\mid x\in \Lb, y\in \Gm, y=x+a_{2^k}\}=\Lb+2a_{2^k}.$$Letting$$C=\{2y-x\mid x\in \Lb, y\in \Gm, y<x+a_{2^k}\},$$we see that all elements of $C$ are less than $$2a_{2^{k+1}-1}-a_{2^k-1}=2a_{2^k}+a_{2^k-1}<3a_{2^k}.$$Hence, $C\subseteq [0,3a_{2^k})$.

Summing up our results, we find that the integers covered by $\Lb\cup \Gm$ are exactly the union \begin{eqnarray*}&&B\cup (a_{2^k-1},a_{2^k})\cup \left(O(A)+a_{2^k}\right)\cup \left(B+a_{2^k}\right)\cup (a_{2^{k+1}-1},2a_{2^k})\cup \left(O(A)+2a_{2^k}\right)\\ &&\cup \left(B+2a_{2^k}\right)\cup (a_{2^k-1}+2a_{2^k},3a_{2^k})\cup \left(O(A)+3a_{2^k}\right)\cup \left(\Lb+2a_{2^k}\right)\cup C.\end{eqnarray*}Restricting to integers greater than $a_{2^{k+1}-1}$, we obtain the set \begin{eqnarray*}&&(a_{2^{k+1}-1},2a_{2^k})\cup \left(O(A)+2a_{2^k}\right)\cup \left(B+2a_{2^k}\right)\cup (a_{2^k-1}+2a_{2^k},3a_{2^k})\\ &&\cup \left(O(A)+3a_{2^k}\right)\cup \left(\Lb+2a_{2^k}\right)\cup(C\cap (a_{2^{k+1}-1},\infty)).\end{eqnarray*}Since the union $\left(O(A)+2a_{2^k}\right)\cup \left(B+2a_{2^k}\right)\cup \left(\Lb+2a_{2^k}\right)$ comprises the entire interval $[2a_{2^k},a_{2^k-1}+2a_{2^k}]$, the preceding expression simplifies to$$\left(a_{2^{k+1}-1},3a_{2^k}\right)\cup \left(O(A)+3a_{2^k}\right)\cup (C\cap (a_{2^{k+1}-1},\infty)).$$Because $C$ is a subset of $[0,3a_{2^k})$, the last term is already included in the first, giving $$\left(a_{2^{k+1}-1},3a_{2^k}\right)\cup \left(O(A)+3a_{2^k}\right).$$

This shows that $$a_{2^{k+1}}=3a_{2^k}=2a_{2^k}+2a_{2^k-1}-\lb+1=2a_{2^{k+1}-1}-\lb+1$$ and more generally that the terms of $S(A)$ that follow $a_{2^{k+1}-1}$ are exactly the elements of the set $A+3a_{2^k}$, followed by as many terms of $S(A)+3a_{2^k}$ as occur before $2a_{2^{k+1}}-a_0=6a_{2^k}$.  Since the elements of $S(A)+3a_{2^k}$ that occur before $6a_{2^k}$ are exactly the elements $$a_0+3a_{2^k},a_1+3a_{2^k},\ldots,a_{2^{k+1}-1}+3a_{2^k},$$we conclude that equations (\ref{an}) and (\ref{another}) hold with $k+1$ substituted for $k$.  Hence, these equations must hold for all $k'>k$ and $S(A)$ is independent.
\end{proof}

\begin{prop} Let $S(A)=\{a_n\}$ be an independent sequence.  Then, there exists a constant $\alpha$ such that, for $k$ large enough, $$a_{2^k}=\alpha\cdot 3^k.$$
\label{alpha}
\end{prop}
\begin{proof} For sufficiently large $k$,\begin{eqnarray*}a_{2^{k+1}}&=&2\left(a_{2^k}+a_{2^k-1}\right)-\lb+1\\ &=&2\left(a_{2^k}+\frac{1}{2}\left(a_{2^k}+\lb\right)\right)-\lb+1\\ &=&3a_{2^k},\end{eqnarray*}which completes the proof.
\end{proof}

\subsection{Regular sequences}

Having now defined the independent Stanley sequences, we can define the more general class of well-structured Stanley sequences to which they belong.

\begin{Def} We say that a Stanley sequence $S(A)=\{a_n\}$ is \emph{regular} if there exist constants $\lb, \sigma$ and an independent Stanley sequence $\{a'_n\}$ such that
\begin{itemize}
\item The character of $\{a'_n\}$ equals $\lb$.
\item For large enough $k$, the equations\begin{eqnarray}a_{2^k-\sigma+i}&=&a_{2^k-\sigma}+a'_i\label{one}\\ a_{2^k-\sigma}&=&2a_{2^k-\sigma-1}-\lb+1\label{two}\end{eqnarray}hold whenever $0\le i< 2^k$.
\end{itemize}

A Stanley sequence that is not regular will be called \emph{irregular}.
\end{Def}

\begin{ex}
The sequence $\{a_n\}=S(0,1,4)$ is regular with $\lb=0$ and $\{a'_n\}=S(0)$.  As with most regular sequences we will consider, $\sigma=0$.

As with independent sequences, we can break up the sequence into blocks $\Gm_k$ as follows, where the length of each block is a power of 2:

\begin{eqnarray*}\{a_n\}&=&0, 1, \underbrace{\fbox{$4, 5,$}}_{\Gm_1}\underbrace{\fbox{$ 11, 12, 14, 15,$}}_{\Gm_2}\underbrace{\fbox{$31, 32, 34, 35, 40, 41, 43, 44,$}}_{\Gm_3}\ldots\\
\{a'_n\}&=&0,1,\underbrace{\fbox{$3,4,$}}_{\Gm_1}\underbrace{\fbox{$9,10,12,13,$}}_{\Gm_2}\underbrace{\fbox{$27,28,30,31,36,37,39,40,$}}_{\Gm_3}\ldots\end{eqnarray*}

Note that, in moving from the last element of one block to the first element of the next block, the value of the sequence doubles and subtracts $(\lb-1)$.  The next block is then a shift of the corresponding preceding terms in the sequence $\{a'_n\}$.  Note that $$2\cdot a_3-\lb+1=2\cdot 5-0+1=11=a_4$$and $\{11,12,14,15\}=\{0,1,3,4\}+11$.  Likewise, $2\cdot 15+1=31$ and $\{31, 32, 34, 35, 40, 41, 43, 44\}=\{0,1,3,4,9,10,12,13\}+31$.
\end{ex}

\begin{prop} If $S(A)=\{a_n\}$ is regular, then there is a unique choice of constants $\lb,\s$ and independent Stanley sequence $\{a'_n\}$ such that the above definition of regularity is satisfied.
\end{prop}
\begin{proof} Observe that uniqueness of $\s$ implies uniqueness of $\lb$ and $\{a'_n\}$.  Suppose for the sake of contradiction that $\s$ can take on distinct values $\s_1<\s_2$ (for corresponding distinct pairs $(\lb,\{a'_n\})$).  Then, for large enough $k$, $$a_{2^{k+1}-\s_2-1}>a_{2^{k+1}-\s_1}\approx 2a_{2^{k+1}-\s_1-1}>2a_{2^k-\s_2},$$implying that, for each choice of $\{a'_n\}$, $$a'_{2^k}=a_{2^{k+1}-\s_2-1}-a_{2^k-\s_2}>a_{2^k-\s_2}.$$Since $a_{2^k}=\alpha\cdot 3^k$, we know that $\log a_{2^k-\s_2}$ must be asymptotically no more than 3.  However, $$a_{2^{k+1}-\s_2}\approx 2a_{2^{k+1}-\s_2-1}>3a_{2^k-\s_2},$$implying that $\log a_{2^k-\s_2}$ is strictly greater than 3, a contradiction.  Hence, $\s$ is unique, implying the proposition.
\end{proof}

For $S(A)$ regular, we write $\lb$, $\s$, and $\{a'_n\}$ as $\lb(A)$, $\s(A)$, and $S'(A)$, respectively and refer to them as the \emph{character}, \emph{shift index}, and \emph{core} of the Stanley sequence $S(A)$.  We say that an integer $k_0$ is \emph{adequate} if (i) all $k\ge k_0$ satisfy equations (\ref{one}) and (\ref{two}), and (ii) $k_0$ is large enough that $a_{2^{k_0}-\sigma(A)}$ is not contained in the minimal nucleating set of $S(A)$.

It is evident that the independent Stanley sequences $S(A)$ are exactly the regular Stanley sequences that satisfy $\s(A)=0$ and $S'(A)=S(A)$.  We say that a sequence is \emph{dependent} if it is regular but not independent.

\begin{prop} Let $S(A)=\{a_n\}$ be a regular sequence, and let $\alpha$ be the constant implied in Proposition \ref{alpha} such that $a'_{2^k}=\alpha\cdot 3^k$ for large $k$.  Then, there exists a constant $\beta$ such that, for $k$ large enough, $$a_{2^k-\sigma(A)}=\alpha \cdot 3^k+\beta \cdot 2^k.$$
\label{beta}
\end{prop}
\begin{proof} Let $\lb=\lb(A)$, $\s=\s(A)$, and $\{a'_n\}=S'(A)$.  Pick some adequate $k$.  Observe that \begin{eqnarray*}a_{2^{k+1}-\s}-2a_{2^k-\s}&=&(2a_{2^{k+1}-\s-1}-\lb+1)-2a_{2^k-\s}\\ &=&2(a_{2^{k+1}-\s-1}-a_{2^k-\s})-\lb+1\\ &=&2a'_{2^k-\s-1}-\lb+1\\ &=&a'_{2^k-\s}\\ &=&\alpha\cdot 3^k,\end{eqnarray*}which proves the proposition.
\end{proof}

This proposition allows us to define the functions $\alpha(A)$ and $\beta(A)$ for each $A$ such that $S(A)$ is regular.  Note that while $\alpha(A)$ and $\beta(A)$ must evidently be rational, they need not be integers, as in the case of $A=\{0,1,7\}$, where $\alpha(A)=10/9$.  It is clear that $\alpha(A)$ must be positive; a similar condition on $\beta(A)$ appears true from data.

\begin{conj} $\beta(A)\ge 0$ for all regular Stanley sequences $S(A)$.
\end{conj}

As a corollary to Proposition \ref{beta}, we obtain the following welcome result.

\begin{cor} All regular Stanley sequences follow Type 1 growth.
\end{cor}

Indeed, our investigation of Stanley sequences suggests that the dichotomy between regular and irregular sequences corresponds precisely with the dichotomy hypothesized in \cite{stanley} between Type 1- and Type 2-growth sequences.

\begin{conj} All irregular Stanley sequences follow Type 2 growth.
\end{conj}

We also mention a useful property which appears to hold for all regular sequences.

\begin{Def} Let $S(A)=\{a_n\}$ and $S(A')=\{a'_n\}$ be Stanley sequences.  We say that $S(A)$ is \emph{faithful} to $S(A')$ if, for each $a'_n<\om(A')$ there exists some $m$ for which $a_m=a'_n$.
\end{Def}

\begin{conj} Every regular sequence is faithful to its core.
\end{conj}

\subsection{The character}

We conclude this section with a consideration of the range of the character function.

\begin{prop} Let $S(A)$ be a regular Stanley sequence.  Then $\lb(A)\ge 0$, with $\lb\ne 1,3$.
\end{prop}
\begin{proof}
Let $\lb(A)=\lb$.  We may assume without loss of generality that $S(A)$ is independent, since the core of $S(A)$ has the same character as $S(A)$ itself.  Then, consider some adequate $k$, so that $$a_{2^k}=2a_{2^k-1}-\lb+1$$holds.  We note that since $a_{2^k}-1=2a_{2^k-1}-\lb$ is not in $S(A)$, it must be covered by the set $T=\{a_0,a_1,\ldots,a_{2^k-1}\}$ and hence can be at most $2a_{2^k-1}$.  We conclude that $\lb\ge 0$.  Further, we note that since $2a_{2^k-1}$ is certainly covered by $T$, the character $\lb$ cannot be $1$.

Suppose for the sake of contradiction that $\lb=3$.  If $1\in S(A)$, then $a_{2^k+1}=a_{2^k}+a_1=a_{2^k}+1$ by regularity.  Because $a_{2^k}=2a_{2^k-1}-2$, we conclude that $$a_{2^k+1}=2a_{2^k-1}-1,$$which is a contradiction since then $1,a_{2^k-1},a_{2^k+1}$ form an arithmetic progression.  We conclude that $1\not\in S(A)$, which means that $2a_{2^k-1}-1$ must be covered by $T$.  Suppose $2t-s=2a_{2^k-1}-1$ for $s,t\in T$.  Since the greatest element of $T$ is $a_{2^k-1}$, we must have $t=a_{2^k-1}$, because smaller $t$ would force $s$ to be negative.  But then $s=1$, which is a contradiction, since we know $1\not\in S(A)$.  We conclude that $\lb\ne 3$.
\end{proof}

For further investigation of forbidden character values, the following lemma is useful.

\begin{lem} If $S(A)$ is independent, then $\om(A)<\lb(A)$.
\label{omlb}
\end{lem}
\begin{proof} Take some extremely large integer $k$ and let $T=\{a_0,a_1,\ldots,a_{2^k-1}\}$.  Let $x=\om(A)+a_{2^k}$.  We know $x$ is covered by $S(A)$, so let $s,t\in S(A)$ be such that $s<t$ and $2t-s=x$.   If neither $s$ nor $t$ is in $T$, then $s'=s-a_{2^k}$ and $t'=t-a_{2^k}$ must be in $S(A)$ and must satisfy $2t'-s'=x-a_{2^k}$.  Since $x-a_{2^k}\in O(A)$ and thus cannot be covered by $S(A)$, this is impossible, so at least one of $s,t$ must be in $T$.  If only $s$ is in $T$, then $$2t-s\ge 2a_{2^k}-a_{2^k-1}=a_{2^k}+a_{2^k-1}-\lb(A)+1,$$which is larger than $x$ because $k$ is large.  Hence, both $s,t$ must be in $T$.  Since the maximum integer covered by $T$ is $2a_{2^k-1}=a_{2^k}+\lb(A)-1$,  the lemma follows.
\end{proof}

\begin{cor} At most finitely many independent Stanley sequences exist with a given character $\lb$.
\end{cor}
\begin{proof} Suppose $S(A)=\{a_n\}$ is independent with $\lb=\lb(A)$ and $\om=\om(A)$, such that $A=\{a_0,a_1,\ldots,a_m\}$ is the minimal nucleating set for the sequence.  Since $A$ is minimal, $\om>a_{m-1}$.  Also, we must have $\om=a_m-1$ unless $a_m-1$ is itself covered by $A$, which can only occur if $a_m-1\le 2a_{m-1}$.  Since $a_{m-1}$ is bounded above by $\om$, this implies that $a_m$ is bounded above by $2\om+1$.

Now, the preceding lemma tells us that $\om<\lb$.  Hence, $a_m\le 2\lb-1$, implying the desired result.
\end{proof}

This corollary tells us that whether or not a given character is possible for an independent (and hence regular) $S(A)$ can be ascertained by checking a finite number of potential nucleating sets $A$.  We have examined (by computer) these possible nucleating sets for many character values; our data suggest that $1,3,5,9,11,15$ are impossible for the character function.  (However, this result is not certain since it assumed the irregularity of various Stanley sequences, while as yet no Stanley sequence has been shown definitively to be irregular.)  For all other characters up to 76, we have found corresponding regular sequences.  (See the appendix for sample data.)  A method we will outline in the next section suggests that all sufficiently large values are possible for the character function.  We therefore offer the following conjecture.

\begin{conj} The range of the character function is exactly the set of integers $n$ that are at least $0$ and are not in the set $\{1,3,5,9,11,15\}$.
\label{characterrange}
\end{conj}

We may also obtain another corollary to the preceding lemma.

\begin{cor} Every regular sequence of character $0$ has $S(0)$ as its core.
\end{cor}
\begin{proof} If $S(A)$ is an independent sequence with character $0$, then $\om(A)<0$ by the lemma, implying that $\om(A)$ is not defined and so $O(A)$ is empty.  Hence, $S(A)=S(0)$ and the result follows.
\end{proof}

\section{Constructing independent sequences}
\label{sec:indep}

Heretofore, the only sequences shown to follow Type 1 growth have been the sequences $S(0,3^k)$ and $S(0,2\cdot 3^k)$, for which complete descriptions were given in \cite{stanley}.  It is easily checked that these sequences are independent for any $k$.  In this section we offer several novel methods for constructing independent sequences, while in the next section we construct dependent sequences. The classes of sequences we describe include $S(0,3^k)$ and $S(0,2\cdot 3^k)$ as special cases.

For convenience in stating certain results, we define the functions $t_i$ on nonnegative integers $x$ by letting $t_i(x)$ equal the digit in the $3^i$s place in the ternary representation of $x$.  Recall Theorem \ref{thm:monotone}, which we restate here.

\begin{thm:monotone} Let $k$ be a positive integer and $\mathcal{A}$ be a monotone decreasing family of subsets of $\{0,1,\ldots,k-1\}$ (i.e., every set in $\mathcal{A}$ has all its subsets contained in $\mathcal{A}$).  Let $$A=\{3^{a_1}+3^{a_2}+\cdots+3^{a_n}\mid \{a_1,a_2,\ldots,a_n\}\in \mathcal{A}\}.$$Then, $S(A\cup \{3^k\})$ and $S(A\cup \{2\cdot 3^k\})$ are independent Stanley sequences.
\end{thm:monotone}

In particular, these sequences admit the following closed-form descriptions:

\begin{enumerate} 
\item $S(A\cup \{3^k\})$ contains exactly those integers $x\ge 0$ such that
\begin{itemize}
\item $t_i(x)=0$ or $1$ for $i\ne k$.
\item If $t_k(x)=0$, then $\sum_{i=0}^{k-1} t_i(x)3^i\in A$.
\item If $t_k(x)=2$, then $\sum_{i=0}^{k-1} t_i(x)3^i\not\in A$.
\end{itemize}
\item $S(A\cup \{2\cdot 3^k\})$ is contains exactly those integers $x\ge 0$ such that
\begin{itemize}
\item $t_i(x)=0$ or $1$ for $i\ne k,k+1$.
\item $t_k(x)=0$ or $2$.
\item If $t_k(x)=t_{k+1}(x)=0$, then $\sum_{i=0}^{k-1} t_i(x)3^i\in A$.
\item If $t_{k+1}(x)=2$, then $t_k(x)=0$ and $\sum_{i=0}^{k-1} t_i(x)3^i\not\in A$.
\end{itemize}
\end{enumerate}

\begin{ex}
\label{ex:monotone}
Take $k=3$ and $\mathcal{A}=\{\emptyset,\{0\},\{1\},\{2\},\{0,2\}\}$.  Then, \begin{eqnarray*}A&=&\{0,1,10,100,101\}\hskip .1 in \text{in base 3}\\ &=&\{0,1,3,9,10\}\hskip .1 in \text{in base 10}.\end{eqnarray*}The theorem implies that $S(0,1,3,9,10,27)$ and $S(0,1,3,9,10,54)$ are independent.  Indeed, \begin{eqnarray*}S(0,1,9,10,27)&=& 0,1,10,100,101,1000,1001,1010,1011,1100,1101,1110,1111,2011,2110,\\ &&2111,10000,10001,10010,10100,10101,11000,11001,11010,11011,11100,\\ &&11101,11110,11111,12011,12110,12111,100000,\ldots\hskip .1 in \text{in base 3}\\ &=& 0, 1, 3, 9, 10, 27, 28, 30, 31, 36, 37, 39, 40, 58, 66, 67,\\ && \underbrace{\fbox{$81, 82, 84,
90, 91, 108, 109, 111, 112, 117, 118, 120, 121, 139, 147, 148,$}}_{\Gm_4}\\ && 243,\ldots \hskip .1 in \text{in base 10}\end{eqnarray*}is independent with character $\lb=54$ satisfying $2\cdot 67-\lb+1=81$ and $2\cdot 148-\lb+1=243$.
\end{ex}

\begin{proof} We will prove the theorem for $S(A\cup \{3^k\})$ (the proof for $S(A\cup \{2\cdot 3^k\})$ is very similar).  Pick some $k$ and $\mA$ according to the theorem statement, let $A$ be defined from $\mA$ as in the theorem, and let $S$ be the sequence consisting of those nonnegative integers $x$ which satisfy the three desired conditions on ternary digits.  We must prove that $S=S(A\cup \{3^k\})$, for which we need (i) that $S$ is 3-free, and (ii) that $x>3^k$ is covered by $S$ if $x\not\in S$.

We first prove (i).  Suppose for the sake of contradiction that there exist $x,y,z\in S$ with $y,z<x$ such that $2y-z=x$.  Since the ternary digits $t_0$ through $t_{k-1}$ must be either 0 or 1 in $x,y,z$, we can conclude that these digits are all the same for $x,y,z$.  Now, if $t_k(x),t_k(y),t_k(z)$ are not to be identical, they must take on all values 0,1,2 in some order.  However, if $t_k$ is 0, the previous ternary digits must form an element of $A$, whereas if $t_k$ is 2, the previous ternary digits cannot form an element of $A$. Since we know that $t_i(x)=t_i(y)=t_i(z)$ for $0\le i\le k-1$, we conclude that $t_k$ is identical for $x,y,z$.  Now, since $t_i$ must be 0 or 1 for $i>k$, every such digit must also be identical for $x,y,z$, implying $x=y=z$, a contradiction. We conclude that $S$ must be 3-free.

We now prove (ii).  Suppose that $x>3^k$ with $x\not\in S$.  We construct $y,z\in S$ digit-wise so that $y,z<x$ and $x=2y-z$.  For each $i<k$, we set
\begin{itemize}
\item $t_i(y)=t_i(z)=0$ if $t_i(x)=0$.
\item $t_i(y)=t_i(z)=1$ if $t_i(x)=1$.
\item $t_i(y)=1$ and $t_i(z)=0$ if $t_i(x)=2$.
\end{itemize}

Before assigning the remaining digits $t_i(y)$ and $t_i(z)$, we define the numbers $y_0$ and $z_0$ to be the ternary subwords of $y$ and $z$, respectively, formed by considering only digits 0 through $k-1$.  We note that the nonzero digits of $z_0$ are a subset of those of $y_0$.  Hence, if $y_0$ is in $A$ then $z_0$ is also, since $\mA$ is monotone decreasing.

This observation made, we now proceed to define the remaining digits.

\begin{case} $t_k(x)\ne 0$.
\end{case}

We begin by assigning $t_i(y)$ and $t_i(z)$ for $i>k$ following the same rules as for $i<k$.  Next, we define $t_k(y)$ and $t_k(z)$, as follows.

If $t_k(x)=1$, we set $t_k(y)=t_k(z)=1$.  By the definition of $S$, the $y$ and $z$ thus constructed will be in $S$, showing $x$ is covered by $S$.  If $t_k(x)=2$ and $z_0\in A$, then we set $t_k(z)=0$ and $t_k(y)=1$.  Again $y,z\in S$, so $x$ is covered by $S$.  On the other hand, if $t_k(x)=0$ and $z_0\not\in A$, then we may conclude $y_0\not\in A$.  We here set $t_k(y)=t_k(z)=2$, and conclude again that $y,z\in S$.

\begin{case} $t_k(x)=0$ and $y_0\in A$.
\end{case}

We begin by assigning $t_i(y)$ and $t_i(z)$ for $i>k$ following the same rules as for $i<k$.  Next, we set $t_k(y)=t_k(z)=0$.  Since $y_0$ is in $A$, $z_0$ must be as well, so $y,z\in S$, as desired.

\begin{case} $t_k(x)=0$ and $y_0\not\in A$
\end{case}

We begin by assigning $t_i(y)$ and $t_i(z)$ for $i>k$ following the same rules as for $i<k$, except with $t_i(x)$ replaced by $t_i(x-3^{k+1})$ throughout.  (Since $x>3^k$ and $t_k(x)=0$, we know that $x-3^{k+1}$ is a nonnegative integer.)  Next, we set $t_k(y)=2$ and $t_k(z)=1$.  It is simple to verify that $y,z\in S$.

We conclude that in all cases $y,z\in S$ and hence all $x\not\in S$ satisfying $x>3^k$ are covered by $S$.  Hence, $S=S(A\cup \{3^k\})$, as desired.  That $S(A\cup \{3^k\})$ is independent follows routinely from the definition of $S$.
\end{proof}

\begin{rem} The characters of $S(A\cup \{3^k\})$ and $S(A\cup \{2\cdot 3^k\})$ can easily be shown to equal $2\cdot 3^k$ and $4\cdot 3^k$, respectively, provided that $\mathcal{A}$ does not contain \emph{all} subsets of $\{0,1,\ldots,k-1\}$.   If $\mathcal{A}$ does contain all subsets of $\{0,1,\ldots,k-1\}$, then the sequence $S(A\cup \{2\cdot 3^k\})$ has character $2\cdot 3^k$, whereas $S(A\cup \{3^k\})$ is simply $S(0)$ and has character $0$ for any $k$.
\end{rem}

We can also develop the preceding theorem in a different way.  Given an independent sequence $S(A)=\{a_n\}$, define the \emph{$k$-reversal} $R_k(A)$ of $S(A)$ as follows: For $x=a_{2^k-1}$, set $$R_k(A)=S(x-a_{2^k-1},x-a_{2^k-2},\ldots,x-a_1,x-a_0).$$Note that this nucleating set is indeed 3-free and starts with 0.  We say that an independent sequence is \emph{reversible} if for every adequate $k$, the $k$-reversal of the sequence is independent.

\begin{prop} As defined in the preceding theorem, the sets $A\cup \{3^k\}$ and $A\cup \{2\cdot 3^k\}$ are reversible.
\end{prop}

We omit the proof of this result, since it is routine and not especially instructive.

Theorem \ref{thm:product} offers a more interesting way of generating new regular Stanley sequences from existing ones.

\begin{thm:product} Let $S(A)=\{a_n\}$ be independent and $S(B)=\{b_n\}$ be regular.  Let $k$ be adequate with respect to $A$.  Let $A^{\ast}=\{a_0,a_1,\ldots,a_{2^k-1}\}$ and define $$A\otimes_k B=\{a_{2^k}b+a\mid  a\in A^{\ast},b\in B\}.$$Then, $S(A\otimes_k B)$ is a regular Stanley sequence, independent if and only if $B$ is, having description $$S(A\otimes_k B)=\{a_{2^k}b+a\mid  a\in A^{\ast},b\in S(B)\},$$with character $\lb(A\otimes_k B)=a_{2^k}\cdot \lb(B)+\lb(A)$ and shift index $\sigma(A\otimes_k B)=2^k\cdot \sigma(B)$.
\end{thm:product}

\begin{ex}
\label{ex:product}
Take $A=\{0\}$ and $B=\{0,2,5\}$.  Then, $a_2=3$ and \begin{eqnarray*}A\otimes_2 B&=&\{3\cdot b+a\mid a\in \{0,1\},\, b\in \{0,2,5\}\\ &=&\{0,1,6,7,15,16\}.\end{eqnarray*}Then, the sequence \begin{eqnarray*}S(0,1,6,7,15,16)&=&0, 1, 6, 7, 15, 16, 18, 19, \underbrace{\fbox{$27, 28, 33, 34, 42, 43, 45, 46,$}}_{\Gm_3}\\ &&\underbrace{\fbox{$81, 82, 87,
88, 96, 97, 99, 100, 108, 109, 114, 115, 123, 124, 126, 127$}}_{\Gm_4},\\ && 243,\ldots\end{eqnarray*}is independent with character $3\lb(B)+\lb(A)=3\cdot 4+0=12$.
\end{ex}

\begin{proof} Let $S=\{a_{2^k}b+a\mid  a\in A^{\ast}, b\in S(B)\}$ be the proposed form of the sequence $S(A\otimes_k B)$, and let $x_0$ be the largest element of $A\otimes_k B$.  It suffices to show (i) that $S$ is 3-free, and (ii) that every integer $x>x_0$ not in $S$ is covered by $S$.

We first prove that $S$ is 3-free.  Suppose for the sake of contradiction that $x,y,z\in S$ exist with $y,z<x$ and $2y-z=x$.  Since no three distinct elements of $A^{\ast}$ form an arithmetic progression modulo $a_{2^k}$, we conclude that $x,y,z$ must all be identical modulo $a_{2^k}$ to some common $a_i$.  But then the elements $(x-a_i)/a_{2^k},(y-a_i)/a_{2^k},(z-a_i)/a_{2^k}$ of $S(B)$ must form an arithmetic progression - a contradiction.  Hence, $S$ must be 3-free.

Now suppose that $x>x_0$ is not in $S$.  We must show it is covered by $S$.  Let $m,r$ be such that $x=m\cdot a_{2^k}+r$.  There are two possibilities:

\setcounter{case}{0}
\begin{case} $r$ is in $A^{\ast}$ or is covered by it.
\end{case}
Pick $a_i,a_j\in A^{\ast}$ such that $2a_i-a_j=r$.  Since $x>x_0$, $m$ must either be in $S(B)$ or else be covered by it.  Picking $b_g,b_h$ such that $2b_g-b_h=m$, we see that $$x=2\left(a_{2^k}b_g+a_i\right)-\left(a_{2^k}b_h+a_j\right).$$

\begin{case} $r\in O(A)$.
\end{case}
In this case, there exist $a_i,a_j\in A^{\ast}$ such that $2a_i-a_j=a_{2^k}+r$.  Then, either $m-1$ is in $S(B)$ or covered by it.  Picking $b_g,b_h$ such that $2b_g-b_h=m-1$, we see that $$x=2\left(a_{2^k}b_g+a_i\right)-\left(a_{2^k}b_h+a_j\right).$$

We conclude that, in both possible cases, $x$ is covered by $S$ and hence that $S$ is indeed $S(A\otimes_k B)$.  That $S(A\otimes_k B)$ is regular, with character and shift index as stated, follows routinely from the explicit description of $S$.
\end{proof} 

\begin{rem} Theorem \ref{thm:product} proves that a great number of integers are attainable as characters of regular Stanley sequences.  For example,\begin{eqnarray*}\lb(\{0\}\otimes_1 A)&=&3\lb(A)\\ \lb(\{0,2\}\otimes_1 A)&=&3\lb(A)+2\end{eqnarray*}It appears possible that similar reasoning could show the attainability of all character values above a certain constant; more research in this area is called for.
\end{rem}

We will refer to the operation $A\otimes_k B$ just described as the \emph{$k$-product} of $A$ and $B$.  We note that $k$-multiplication of independent sequences is associative; this follows immediately from our closed-form description of the terms of $A\otimes_k B$.

Theorem \ref{thm:product} allows for the construction of many new regular Stanley sequences, such as those described in Theorem \ref{thm:digits}.

\begin{thm:digits} Let $k$ be a positive integer. Let $T_1,T_2$ be disjoint subsets of $\{0,1,\ldots,k\}$ such that no $t\in T_1$ satisfies $t-1\in T_2$.  Let $$A=\left\{\left(3^{a_1}+3^{a_2}+\cdots+3^{a_m}\right)+2\left(3^{b_1}+3^{b_2}+\cdots+3^{b_n}\right)\mid \{a_1,\ldots,a_m\}\subseteq T_1,\{b_1,\ldots,b_n\}\subseteq T_2\right\}.$$Then, $S(A)$ is an independent Stanley sequence.
\end{thm:digits}

\begin{ex}
\label{ex:digits}
Let $k=3$, let $T_1=\{0,3\}$ and $T_2=\{1\}$.  Then, \begin{eqnarray*}A&=&\{0,1,20,21,1000,1001,1020,1021\}\hskip .1 in \text{in base 3}\\ &=&\{0,1,6,7,27,28,33,34\}\hskip .1 in \text{in base 10}.\end{eqnarray*}

Moreover, the sequence \begin{eqnarray*}S(A)&=&0,1,20,21,1000,1001,1020,1021,1100,1101,1120,1121,2100,2101,2120,2121,\\ && 10000,10001,10020,10021,11000,11001,11020,11021,11100,11101,11120,11121,\\ && 12100,12101,12120,12121,\\ &&100000,\ldots\hskip .1 in \text{in base 3}\\ &=&0, 1, 6, 7, 27, 28, 33, 34, 36, 37, 42, 43, 63, 64, 69, 70, \\ &&\underbrace{\fbox{$81, 82, 87,
88, 108, 109, 114, 115, 117, 118, 123, 124, 144, 145, 150, 151,$}}_{\Gm_4}\\ &&243,\ldots \hskip .1 in \text{in base 10}\end{eqnarray*}is independent with character $\lb=60$.
\end{ex}

Before proving Theorem \ref{thm:digits}, we shall present two weaker versions of this theorem.  Their proofs follow routine case analysis and are omitted.

\begin{lem} Let $k$ be a positive integer, and $j$ an integer such that $0\le j\le k$.  Set $$B_1=\{3^{b_1}+3^{b_2}+\cdots+3^{b_m}\mid j\le b_1<b_2<\cdots <b_m\le k\}.$$Then, $S(B_1)$ is independent, with $k+1$ adequate.  Specifically, $S(B_1)$ consists of all integers $x\ge 0$ satisfying
\begin{itemize}
\item $t_i(x)$ equals 0 or 1 for all $i$ not in the interval $[j,k]$.
\item If $t_j(x),t_{j+1}(x),\ldots,t_k(x)$ are all $0$ or $1$, but are not all 1, then $t_i=0$ for all $i<j$.
\item If $t_j(x),t_{j+1}(x),\ldots,t_k(x)$ are not all either $0$ or $1$, then (i) not all $t_i(x)$ are 0 for $i<j$, and (ii) $t_i(x)$ equals 1 or 2 for $j\le i\le k$.
\end{itemize}
\label{lem1}
\end{lem}

\begin{lem} Let $k$ be a positive integer, and $j$ an integer such that $0\le j\le k$.  Set $$B_2=\{2\left(3^{b_1}+3^{b_2}+\ldots+3^{b_m}\right)\mid j\le b_1<b_2<\cdots <b_m\le k\}.$$  Then, $S(B_2)$ is independent, with $k+2$ adequate.

To describe the elements of $B_2$ in closed form, we first define a function $\zeta$ on the set of nonnegative integers $x$ such that $t_i(x)$ equals 0 or 2 for all $j\le i\le k$.  Set $\zeta(x)$ equal to
\begin{itemize}
\item $x$ itself, if $t_i(x)=0$ for all $i<j$
\item otherwise, the integer obtained from $x$ be switching to 1 all digits $t_i(x)=2$ such that $j\le i\le k$ and at least one of $t_j(x),t_{j+1}(x),\ldots,t_{i-1}(x)$ is zero.
\end{itemize}

Then, $S(B_2)$ consists of all integers $\zeta(x)$ such that $x$ satisfies

\begin{itemize}
\item $t_i(x)$ equals 0 or 2 for all $i$ in the interval $[j,k+1]$.
\item $t_i(x)$ equals 0 or 1 for all $i$ not in the interval $[j,k+1]$.
\item If $t_j(x),t_{j+1}(x),\ldots,t_k(x)$ are all $0$ or $2$, but are not all 2, then $t_i=0$ for all $i<j$.
\item If $t_{k+1}=0$, then $t_j(x),t_{j+1}(x),\ldots,t_k(x)$ are all $0$ or $2$.
\item If $t_{k+1}=2$, then $x'\not\in S$, where $x'$ is obtained from $x$ by switching $t_{k+1}$ to $0$.
\end{itemize}
\label{lem2}
\end{lem}

\begin{proof}[Proof of Theorem \ref{thm:digits}]
We observe that $A$ can be expressed as the product $$A_1\otimes_{k_1}A_2\otimes_{k_2}\cdots \otimes_{k_{m-1}}A_m,$$ where each $A_i$ is either of the form $B_1$ (see Lemma \ref{lem1}) or of the form $B_2$ (see Lemma \ref{lem2}), with each $k_i$ a corresponding adequate integer to $A_i$.  Then, applying Theorem \ref{thm:product} finishes the proof.
\end{proof}

\section{Constructing dependent sequences}
\label{sec:dep}

In this section we will demonstrate two methods for constructing dependent sequences from existing regular sequences.

\subsection{Shifted Stanley sequences}

Given the regular sequence $S(A)=\{a_n\}$ and the nonnegative integers $k$ and $c$, let the \emph{$k$-shifted Stanley sequence} $S_k(c,A)$ be the Stanley sequence generated by the set $$A_k(c,A)=\left\{a_i\mid 0\le i<2^k-\sigma(A)\right\} \cup \left\{c+a_i\mid 2^k-\sigma(A)\le i<2^{k+1}-\sigma(A)\right\},$$assuming that this set is 3-free (if it is not, then $S_k(c,A)$ is not defined).

For regular sequences $S(A)$, we will use the notation $O'(A)$ to denote the omitted set of $S'(A)$. Recall Theorem \ref{thm:shift}:

\begin{thm:shift} Let $S(A)=\{a_n\}$ be an independent sequence with character $\lb$.  Let $\ell$ be the minimum adequate integer for $S(A)$, and pick $k\ge \ell$.  Let $c$ be such that \begin{equation}\lb\le c\le a_{2^k-2^\ell}-\lb.\label{bounds}\end{equation}Then, $S_k(c,A)$ is defined and is a regular Stanley sequence with core $S(A)$.
\end{thm:shift}

We also conjecture the following stronger statement.

\begin{conj}
Let $S(A)=\{a_n\}$ be a regular sequence with core $\{a'_n\}$, shift index $\sigma$, and character $\lb$.  Let $k$ be an adequate integer such that for all $0\le i<2^{k-1}$,\begin{eqnarray*}a_{2^{k-1}-\s+i}&=&a_{2^{k-1}-\s}+a'_i.\end{eqnarray*}Let $\ell$ be the minimal adequate integer for $\{a'_n\}$.  Let $c$ be such that \begin{equation}\lb\le c\le a_{2^k-2^\ell-\sigma}+a'_{2^k}-a_{2^k-\sigma}-\lb.\end{equation}Then, $S_k(c,A)$ is defined and is a regular Stanley sequence with core $S(A)$.
\label{conjshifts}
\end{conj}

\begin{ex}
\label{ex:shift}
Let $\{a_n\}=S(0)$, and $k=2$.  Then, $\ell=0$ and $\{a'_n\}=\{a_n\}$, because $S(0)$ is independent.  Theorem \ref{thm:shift} implies that $S_2(c,\{0\})$ is defined for all $c$ such that $$0\le c\le a_3+a'_4-a_4-\lb=4+9-9-0=4.$$Picking $c=3$, we compute $S_2(3,\{0\})$ by taking the highlighted block of $S(0)$, adding 3 to the block, then recomputing the subsequent terms of the sequence.\begin{eqnarray*}S(0)&=&0,1,\underbrace{\fbox{$3,4,$}}_{\Gm_1}\underbrace{\fbox{${\bf 9,10,12,13,}$}}_{\Gm_2}\underbrace{\fbox{$27,28,30,31,36,37,39,40,$}}_{\Gm_3}\\ && \underbrace{\fbox{$81,82, 84, 85, 90, 91, 93, 94, 108, 109, 111, 112, 117, 118, 120, 121,$}}_{\Gm_4}\ldots.\end{eqnarray*}

This results in the sequence \begin{eqnarray*}&&S(0,1,3,4,12,13,15,16)\\ &=&0, 1, 3, 4, \underbrace{\fbox{$12, 13, 15, 16,$}}_{\Gm_2}\underbrace{\fbox{$33, 34, 36, 37, 42, 43, 45, 46,$}}_{\Gm_3}\\ &&\underbrace{\fbox{$93, 94, 96, 97, 102, 103, 105, 106, 120, 121, 123, 124, 129, 130, 132, 133,$}}_{\Gm_4}\ldots.\end{eqnarray*}

Note this sequence is dependent, with core $S(0)$.  In fact, it is possible to construct $S_2(3,\{0\})$ from $S(0)$ as follows: Add 3 to the block $\{9,10,12,13\}$, then add $2\cdot 3$ to the next block, add $2^2\cdot 3$ to the next block, etc.

We now repeat the shifting process on the block $\{33, 34, 36, 37, 42, 43, 45, 46\}$ of the sequence $S_2(3,\{0\})$.  According to Conjecture \ref{conjshifts}, we are able to pick any $c$ such that $$\lb\le c\le 2\cdot 13-16=10,$$where 16 is the preceding element of the sequence $S_2(3,\{0\})$ and $13$ is the corresponding element of the core sequence $S(0)$.  Picking $c=10$, we have \begin{eqnarray*}&&S(0, 1, 3, 4, 12, 13, 15, 16, 43, 44, 46, 47, 52, 53, 55, 56)\\ &=&0, 1, 3, 4, 12, 13, 15, 16, \underbrace{\fbox{$43, 44, 46, 47, 52, 53, 55, 56,$}}_{\Gm_3}\\ &&\underbrace{\fbox{$113, 114, 116, 117, 122, 123, 125, 126, 140, 141, 143, 144, 149, 150, 152, 153,$}}_{\Gm_4}\ldots.\end{eqnarray*}Note that this sequence is dependent, with core $S(0)$.
\end{ex}

To prove Theorem \ref{thm:shift}, we begin with the following lemma.

\begin{lem} Let $S(A)=\{a_n\}$ be independent.  Pick $m$ adequate, and set $\Lb_m=\{a_i\mid 0\le i<2^m\}$.  Let nonnegative integers $d,e$ be such that $a_{2^m-1}+d\le e$ (so that $\Lb_m+d$ and $\Lb_m+e$ occupy disjoint intervals).  Then, $\Lb_m+d$ and $\Lb_m+e$ jointly cover$$\left([2e-d,a_{2^m}+2e-d)\backslash \left(O(A)+(2e-d)\right)\right)\cup \left(O(A)+(a_{2^m}+2e-d)\right).$$
\label{shiftlem}
\end{lem}
\begin{proof} It is obvious that \begin{equation}\{2y-x\mid x\in \Lb_m+d, y\in \Lb_m+e, y=x+(e-d)\}=(\Lb_m+(2e-d))\label{lemone}.\end{equation}Furthermore, since $\Lb_m$ covers $$[0,a_{2^m})\backslash \left(O(A)\cup \Lb_m\right)$$we conclude that \begin{equation}\{2y-x\mid x\in \Lb_m+d, y\in \Lb_m+e, y>x+(e-d)\}\supseteq \left([0,a_{2^m})\backslash \left(O(A)\cup \Lb_m\right)\right)+(2e-d)\label{lemtwo}\end{equation}

Now, consider some large $n$.  It is evident that $O(A)+a_{2^n}$ is not in $S(A)$ and hence must be covered by it.  Pick some element $s\in O(A)+a_{2^n}$ and suppose $2y-x=s$ with $x,y\in S(A)$ and $y>x$.  Since $s\in O(A)+a_{2^n}$, $x<a_{2^n}$.  Then, since $n$ is large, $y<a_{2^n}$.  Because $m$ is adequate, $2a_{2^m-1}-\lb+1=a_{2^m}>a_{2^m-1}$, so $a_{2^m-1}\ge \lb$.  Then, we conclude that \begin{equation}s\ge a_{2^n}=2a_{2^n-1}-\lb+1>2a_{2^n-1}-a_{2^m-1}.\label{sineq}\end{equation}Since $x,y\le a_{2^n-1}$ and $2y-x=s$, it is simple to conclude from (\ref{sineq}) that $x\le a_{2^m-1}$ and $y\ge a_{2^n-1}-a_{2^m-1}=a_{2^n-2^m}$, implying that $x\in \Lb_m$ and $y\in \Lb_m+a_{2^n-2^m}$.  This implies that $\Lb_m$ and $\Lb_m+a_{2^n-2^m}$ jointly cover $O(A)+a_{2^n}$.  Applying the Cover-shift Lemma tells us that $\Lb_m+d$ and $\Lb_m+e$ must jointly cover \begin{eqnarray*}O(A)+a_{2^n}+2(e-a_{2^n-2^m})-d&=& O(A)+a_{2^n}-2(a_{2^n-1}-a_{2^m-1})+(2e-d)\\ &=& O(A)+a_{2^n}-(a_{2^n}+\lb-1)+(a_{2^m}+\lb-1)+(2e-d)\\ &=&O(A)+a_{2^m}+2e-d.\end{eqnarray*}Combining this result with (\ref{lemone}) and (\ref{lemtwo}), we conclude that $\Lb_m+d$ and $\Lb_m+e$ must jointly cover$$\left([2e-d,a_{2^m}+2e-d)\backslash \left(O(A)+(2e-d)\right)\right)\cup \left(O(A)+(a_{2^m}+2e-d)\right),$$as desired.
\end{proof}

We now use Lemma \ref{shiftlem} to prove a stronger version of itself.

\begin{lem}
Let $S(A)=\{a_n\}$ be independent with character $\lb$, and let $\ell$ be the minimal adequate integer for $S(A)$.  Pick $k\ge \ell$ and set $\Lb_k=\{a_i\mid 0\le i<2^k\}$.  Let nonnegative integers $d,e$ be such that $a_{2^k-1}+d\le e$ (so that $\Lb_k+d$ and $\Lb_k+e$ occupy disjoint intervals).  Then, $\Lb_k+d$ and $\Lb_k+e$ jointly cover$$[2e-d-a_{2^k-2^\ell}+\lb,a_{2^k}+2e-d)\cup \left(O(A)+(a_{2^k}+2e-d)\right).$$
\label{shiftlemtwo}
\end{lem}

\begin{proof}
Let $m<k$ be an adequate integer for $\{a_m\}$, and let $\Lb_m=\{a_i\mid 0\le i<2^m\}$.  By Lemma \ref{shiftlem}, \begin{eqnarray}&&\{2(a_j+e)-(a_i+d)\mid 2^k-2^m\le i\le 2^k-1,  0\le j\le 2^m-1\}\nonumber\\ &=&\{2y-x\mid x\in \Lb_m+a_{2^k-2^m}+d, y\in \Lb_m+e\}\nonumber \\ &\supseteq &\left([2e-a_{2^k-2^m}-d,a_{2^m}+2e-a_{2^k-2^m}-d)\backslash \left(O(A)+(2e-a_{2^k-2^m}-d)\right)\right)\nonumber \\ &&\cup \left(O(A)+(a_{2^m}+2e-a_{2^k-2^m}-d)\right)\nonumber \\ &=&\left([2e-a_{2^k-2^m}-d,2e-a_{2^k-2^{m+1}}-d)\backslash \left(O(A)+(2e-a_{2^k-2^m}-d)\right)\right)\label{telescope}\\ &&\cup \left(O(A)+(2e-a_{2^k-2^{m+1}}-d)\right)\nonumber,\end{eqnarray}where the last step follows from \begin{equation}a_{2^k-2^m}-a_{2^m}=a_{2^k-1}-a_{2^m-1}-a_{2^m}=a_{2^k-1}-a_{2^{m+1}-1}=a_{2^k-2^{m+1}}.\label{oncemore}\end{equation}

Hence, the expression given on the right side of (\ref{telescope}) is jointly covered by $\Lb_k+d$ and $\Lb_k+e$.  We take the union of this expression over all possible $m$ ($\ell\le m\le k-1$) and observe that it ``telescopes," becoming the expression \begin{equation}\left([2e-a_{2^k-2^\ell}-d,2e-d)\backslash \left(O(A)+(2e-a_{2^k-2^\ell}-d)\right)\right)\cup \left(O(A)+2e-d\right),\label{combinedone}\end{equation}which must in turn be jointly covered by $\Lb_k+d$ and $\Lb_k+e$.

Again applying Lemma \ref{shiftlem}, we see that $\Lb_k+d$ and $\Lb_k+e$ must also jointly cover \begin{eqnarray}&&\{2(a_j+e)-(a_i+d)\mid 0\le i,j\le 2^k-1\}\nonumber\\ &=&\{2y-x\mid x\in \Lb_k+d, y\in \Lb_k+e\}\nonumber \\ &\supseteq &\left([2e-d,a_{2^k}+2e-d)\backslash \left(O(A)+2e-d)\right)\right)\cup \left(O(A)+(a_{2^k}+2e-d)\right).\label{combinedtwo}\end{eqnarray}

Combining (\ref{combinedone}) and (\ref{combinedtwo}), we see that $\Lb_k+d$ and $\Lb_k+e$ must jointly cover \begin{eqnarray}&&\left([2e-a_{2^k-2^\ell}-d,a_{2^k}+2e-d)\backslash \left(O(A)+(2e-a_{2^k-2^\ell}-d)\right)\right)\cup \left(O(A)+(a_{2^k}+2e-d)\right)\nonumber \\ &\subseteq & [2e-a_{2^k-2^\ell}-d+\lb,a_{2^k}+2e)\cup \left(O(A)+(a_{2^k}+2e)\right),\end{eqnarray}where the last step follows from Lemma \ref{omlb}, since every element of $O(A)$ is at most $\lb$.  This finishes the proof of the lemma.
\end{proof}

\begin{proof}[Proof of Theorem \ref{thm:shift}]
Let $c$ be such that (\ref{bounds}) is satisfied. Let $S_k(c,A)=\{\at_i\}$ be our shifted Stanley sequence.  For each $i$, let \begin{eqnarray*}\Lb_j&=&\{a_i\mid\, 0\le i<2^j\}\\ \Xt_j&=&\{\at_i\mid\, 0\le i<2^j\}\\ \Gm_i &=&\{a_i\mid\, 2^j\le i<2^{j+1}\}\\ \Gmt_i &=&\{\at_i\mid\, 2^j\le i<2^{j+1}\}.\end{eqnarray*}  We claim that $\Gmt_j=\Gm_j+2^{j-k}\cdot c$ for each $j\ge k$.  By construction, we know already that $\Gmt_k$ is of this form.  In our proof, we consider the block $\Gmt_{k+1}$, and then we use induction to prove the result for all $j\ge k+1$.

First, we prove that $S_k(c,A)$ is defined by proving that its nucleating set $A^\ast=\Lb_k\cup (\Gm_k+c)$ is 3-free.  Observe that all elements of $\Gm_k+c$ are at least $$a_{2^k}+c\ge a_{2^k}+\lb\ge 2a_{2^k-1}+1$$and therefore cannot be covered by $\Lb_k$.  Hence, $A^\ast$ is indeed 3-free.

We now consider what is covered by $\Gm_k+c$ alone.  Since $\Lb_k$ covers the set $[a_{2^k-1}+1,a_{2^k})\cup \left(O(A)+a_{2^k}\right)$, we conclude that $\Gm_k+c=\Lb_k+a_{2^k}+c$ must cover the set \begin{equation}[a_{2^k-1}+a_{2^k}+c+1,2a_{2^k}+c)\cup \left(O(A)+2a_{2^k}+c\right)=[a_{2^{k+1}-1}+c+1,2a_{2^k}+c)\cup \left(O(A)+2a_{2^k}+c\right).\label{covers}\end{equation}

We now apply Lemma \ref{shiftlemtwo} to $\Lb_k$ and $\Gm_k+c=\Lb_k+a_{2^k}+c$.  This implies that $\Lb_k$ and $\Gm_k+c$ must jointly cover \begin{equation}[2a_{2^k}+2c-a_{2^k-2^\ell}+\lb,3a_{2^k}+2c)\cup \left(O(A)+(3a_{2^k}+2c)\right).\label{combined}\end{equation}Observe that \begin{eqnarray*}2a_{2^k}+2c-a_{2^k-2^\ell}+\lb&\le &2a_{2^k}+c+\left(a_{2^k-2^\ell}-\lb\right)-a_{2^k-2^\ell}+\lb\\ &=& 2a_{2^k}+c.\end{eqnarray*}Now, we can combine (\ref{covers}) and (\ref{combined}) to conclude that $A^\ast$ covers $$[a_{2^{k+1}-1}+c+1,3a_{2^k}+2c)\cup \left(O(A)+(3a_{2^k}+2c)\right).$$

We can rewrite $$3a_{2^k}=2a_{2^k}+(2a_{2^k-1}-\lb+1)=2(a_{2^k}+a_{2^k-1})-\lb+1=2a_{2^{k+1}-1}-\lb+1=a_{2^{k+1}}.$$Hence, $A^\ast$ covers \begin{equation}[a_{2^{k+1}-1}+c+1,a_{2^{k+1}}+2c)\cup \left(O(A)+(a_{2^{k+1}}+2c)\right).\label{full}\end{equation}

Let $Q$ be the set of the integers that are at least $a_{2^{k+1}}+2c$ and are covered by $A^\ast$.  We claim that \begin{equation}Q=O(A)+(a_{2^{k+1}}+2c).\label{q}\end{equation}  From (\ref{full}), we know that $Q\supseteq O(A)+(a_{2^{k+1}}+2c)$.  We now prove the other direction: $\subseteq$.

Pick some $q\in Q$. Suppose for the sake of contradiction that $q$ is covered by $\Gm_k+c$.  Since the largest element of $\Gm_k+c$ is $a_{2^{k+1}-1}+c$ and the smallest is $a_{2^k}+c$, we know that $$2a_{2^{k+1}-1}-a_{2^k}+c\ge q\ge a_{2^{k+1}}+2c=2a_{2^{k+1}-1}-\lb+1+2c$$and therefore that $c\le \lb-1-a_{2^k}$, an impossibility since $a_{2^k}\ge \lb$.

We conclude that $q$ is not covered by $\Gm_k+c$ and therefore must be jointly covered by $\Lb_k$ and $\Gm_k+c$.  From the original Stanley sequence $S(A)$, we know that $\Lb_k$ and $\Gm_k$ jointly cover no integers greater than $a_{2^{k+1}}$, except for those in the set $O(A)+a_{2^{k+1}}$.  By the Cover-shift Lemma, $\Lb_k$ and $\Gm_k+c$ must jointly cover no integers greater than $a_{2^{k+1}}+2c$, except for those in the set $O(A)+a_{2^{k+1}}+2c$.  Therefore $Q\subseteq O(A)+(a_{2^{k+1}}+2c)$, proving the equation (\ref{q}).

Recall that we wished to prove that the block $\Gmt_{k+1}$ in $S_k(c,A)$ is equal to $\Gm_{k+1}+2c$.  This now follows immediately from (\ref{full}) and (\ref{q}).

Consider some $j\ge k+1$, and assume towards induction that we have $\Gmt_j=\Gm_j+2^{j-k}\cdot c$.  Assume further that $\Xt_{j-1}$ and $\Gmt_{j-1}$ jointly cover the set \begin{equation}[2\at_{2^{j-1}}-a_{2^{j-1}-2^\ell}+\lb,\at_{2^j})\cup \left(O(A)+\at_{2^j}\right).\label{hypothesis}\end{equation}The base case of $j=k+1$ follows from (\ref{combined}).

Now, let \begin{eqnarray*}\Lb^1_j&=&\{\at_i\mid\, 0\le i<2^{j-1}\}\\ \Lb^2_j&=&\{\at_i\mid\, 2^{j-1}\le i<2^j\}\\ \Gm^1_j&=&\{\at_i\mid\, 2^j\le i<2^j+2^{j-1}\}\\ \Gm^2_j&=&\{\at_i\mid\, 2^j+2^{j-1}\le i<2^{j+1}\},\end{eqnarray*}so that $\Xt_j=\Lb^1_j\cup \Lb^2_j$ and $\Gmt_j=\Gm^1_j\cup \Gm^2_j$.

Note that $\Lb^2_j=\Lb_{j-1}+\at_{2^{j-1}}$ and $\Gm^1_j=\Lb_{j-1}+\at_{2^j}$.  Then, applying Lemma \ref{shiftlemtwo}, we conclude that $\Lb^2_j$ and $\Gm^1_j$ jointly cover \begin{eqnarray}&&[2\at_{2^j}-\at_{2^{j-1}}-a_{2^{j-1}-2^\ell}+\lb,a_{2^{j-1}}+2\at_{2^j}-\at_{2^{j-1}})\nonumber\\ &&\cup\left(O(A)+(a_{2^{j-1}}+2\at_{2^j}-\at_{2^{j-1}})\right)\\ &\supseteq& [2\at_{2^j}-\at_{2^{j-1}}-a_{2^{j-1}-2^\ell}+\lb,a_{2^{j-1}}+2\at_{2^j}-\at_{2^{j-1}}).\label{lastone}\end{eqnarray}

Similarly $\Lb^2_j$ and $\Gm^2_j=\Gm^2_j+a_{2^{j-1}}$ must jointly cover \begin{equation} [2\at_{2^j}+2a_{2^{j-1}}-\at_{2^{j-1}}-a_{2^{j-1}-2^\ell}+\lb,3a_{2^{j-1}}+2\at_{2^j}-\at_{2^{j-1}}).\label{lastthree}\end{equation}

By our inductive hypothesis, we know that $\Xt_{j-1}$ and $\Gmt_{j-1}$ jointly cover (\ref{hypothesis}).  Since $\Lb^1_j=\Xt_{j-1}$ and $\Gm^1_j=\Gmt_{j-1}+\at_{2^j}-\at_{2^{j-1}}$, the Cover-shift Lemma implies that $\Lb^1_j$ and $\Gm^1_j$ jointly cover \begin{eqnarray}&&[2\at_{2^{j-1}}-a_{2^{j-1}-2^\ell}+\lb+2(\at_{2^j}-\at_{2^{j-1}}),\at_{2^j}+2(\at_{2^j}-\at_{2^{j-1}}))\nonumber\\ &&\cup \left(O(A)+(\at_{2^j}+2(\at_{2^j}-\at_{2^{j-1}}))\right)\nonumber \\ &\supseteq & [2\at_{2^j}-a_{2^{j-1}-2^\ell}+\lb,3\at_{2^j}-2\at_{2^{j-1}}).\label{lasttwo}\end{eqnarray}

Similarly, $\Lb^1_j=\Xt_{j-1}$ and $\Gm^2_j=\Gm^2_j+a_{2^{j-1}}$ must jointly cover \begin{equation} [2\at_{2^j}+2a_{2^{j-1}}-a_{2^{j-1}-2^\ell}+\lb,3\at_{2^j}+2a_{2^{j-1}}-2\at_{2^{j-1}})\cup \left(O(A)+(3\at_{2^j}+2a_{2^{j-1}}-2\at_{2^{j-1}})\right).\label{lastfour}\end{equation}

Combining (\ref{lastone}), (\ref{lasttwo}), (\ref{lastthree}), and (\ref{lastfour}), we conclude that $\Xt_j$ and $\Gmt_j$ must jointly cover \begin{eqnarray}&&[2\at_{2^j}-\at_{2^{j-1}}-a_{2^{j-1}-2^\ell}+\lb,a_{2^{j-1}}+2\at_{2^j}-\at_{2^{j-1}})\nonumber \\ &\cup& [2\at_{2^j}-a_{2^{j-1}-2^\ell}+\lb,3\at_{2^j}-2\at_{2^{j-1}})\nonumber \\ &\cup &[2\at_{2^j}+2a_{2^{j-1}}-\at_{2^{j-1}}-a_{2^{j-1}-2^\ell}+\lb,3a_{2^{j-1}}+2\at_{2^j}-\at_{2^{j-1}})\nonumber \\ &\cup &[2\at_{2^j}+2a_{2^{j-1}}-a_{2^{j-1}-2^\ell}+\lb,3\at_{2^j}+2a_{2^{j-1}}-2\at_{2^{j-1}})\nonumber \\ &\cup &\left(O(A)+(3\at_{2^j}+2a_{2^{j-1}}-2\at_{2^{j-1}})\right).\label{last}\end{eqnarray}

We note that \begin{eqnarray*}a_{2^{j-1}-2^\ell}-\lb &\ge &2^{j-1-k}(a_{2^k-2^\ell}-\lb)\\ &\ge &2^{j-1-k}c\\ &=&\at_{2^{j-1}}-a_{2^{j-1}}.\end{eqnarray*}Therefore, $$a_{2^{j-1}}+2\at_{2^j}-\at_{2^{j-1}}\ge 2\at_{2^j}-a_{2^{j-1}-2^\ell}+\lb$$and$$3a_{2^{j-1}}+2\at_{2^j}-\at_{2^{j-1}}\ge 2\at_{2^j}+2a_{2^{j-1}}-a_{2^{j-1}-2^\ell}+\lb.$$These two inequalities allow us to simplify (\ref{last}) to \begin{eqnarray}&&[2\at_{2^j}-\at_{2^{j-1}}-a_{2^{j-1}-2^\ell}+\lb,3\at_{2^j}-2\at_{2^{j-1}})\nonumber \\ &\cup &[2\at_{2^j}+2a_{2^{j-1}}-\at_{2^{j-1}}-a_{2^{j-1}-2^\ell}+\lb,3\at_{2^j}+2a_{2^{j-1}}-2\at_{2^{j-1}})\nonumber \\ &\cup &\left(O(A)+(3\at_{2^j}+2a_{2^{j-1}}-2\at_{2^{j-1}})\right).\label{lastagain}\end{eqnarray}

We observe that \begin{eqnarray*}\at_{2^j}-\at_{2^{j-1}}&=&\left(3a_{2^j}+2^{j-k}\cdot c\right)-\left(a_{2^{j-1}}+2^{j-1-k}\cdot c\right)\\ &\ge &2a_{2^{j-1}}\\ &\ge &2a_{2^{j-1}}-a_{2^{j-1}-2^\ell}+\lb.\end{eqnarray*}Therefore, $$3\at_{2^j}-2\at_{2^{j-1}}\ge 2\at_{2^j}+2a_{2^{j-1}}-\at_{2^{j-1}}-a_{2^{j-1}-2^\ell}+\lb.$$Then, we can simplify (\ref{lastagain}) to conclude that $\Xt_j$ and $\Gmt_j$ jointly cover $$[2\at_{2^j}-\at_{2^{j-1}}-a_{2^{j-1}-2^\ell}+\lb,3\at_{2^j}+2a_{2^{j-1}}-2\at_{2^{j-1}})\cup \left(O(A)+(3\at_{2^j}+2a_{2^{j-1}}-2\at_{2^{j-1}})\right).$$We see that \begin{eqnarray*}3\at_{2^j}+2a_{2^{j-1}}-2\at_{2^{j-1}}&=&3\left(a_{2^j}+2^{j-k}c\right)-2^{j-k}\cdot c\\ &=&a_{2^{j+1}}-2^{j+1-k}\cdot c,\end{eqnarray*}so $\Xt_j$ and $\Gmt_j$ jointly cover\begin{equation}[2\at_{2^j}-\at_{2^{j-1}}-a_{2^{j-1}-2^\ell}+\lb,a_{2^{j+1}}-2^{j+1-k}\cdot c)\cup \left(O(A)+(a_{2^{j+1}}-2^{j+1-k}\cdot c)\right).\label{almost}\end{equation}

Finally, we observe that $\Gmt_j=\Lb_j+\at_{2^j}$.  Since $\Lb_j$ covers $[a_{2^j-1}+1,a_{2^j})$, we conclude that $\Gmt_j$ must cover \begin{equation}[\at_{2^j}+a_{2^j-1}+1,\at_{2^j}+a_{2^j}).\label{short}\end{equation}Note that \begin{eqnarray*}\at_{2^{j-1}}+a_{2^{j-1}-2^\ell}-\lb&\ge &2^{j-k}(a_{2^k-2^\ell}-\lb)\\ &\ge &2^{j-k}c\\ &=&\at_{2^j}-a_{2^j}.\end{eqnarray*}This implies that$$\at_{2^j}+a_{2^j}\ge 2\at_{2^j}-\at_{2^{j-1}}-a_{2^{j-1}-2^\ell}+\lb,$$so we can combine (\ref{short}) with (\ref{almost}) to conclude that $\Xt_j\cup \Gmt_j$ must cover $$[\at_{2^j}+a_{2^j-1}+1,a_{2^{j+1}}-2^{j+1-k}\cdot c)\cup \left(O(A)+(a_{2^{j+1}}-2^{j+1-k}\cdot c)\right).$$Hence, $\Gmt_{j+1}=\Gamma_{j+1}+2^{j+1-k}\cdot c$, which, together with (\ref{almost}), completes the induction.

Thus, $\Gmt_j=\Gm_j+2^{j-k}\cdot c$ for each $j\ge k$.  This shows that $S_k(c,A)$ is a regular Stanley sequence, with core $S(A)$.
\end{proof}

\begin{cor} For each nonnegative integer $\lb$, there are either no regular Stanley sequences with character $\lb$, or else infinitely many.
\end{cor}
\begin{proof} If any regular sequence has character $\lb$, then its core must be an independent sequence $\{a_n\}$ with character $\lb$.  The preceding theorem shows that it is possible to construct infinitely many dependent Stanley sequences with $\{a_n\}$ as their core.
\end{proof}

\subsection{Deletions in Stanley sequences}

Finally, we consider the matter of \emph{deletions}.  Erd\H{o}s et al.~\cite{erdos} and Moy \cite{moy} appear to have assumed that Stanley sequences are \emph{maximal} 3-free sets; however, this is not true. For some dependent Stanley sequences, it is possible to remove one or more elements while preserving the Stanley sequence condition.  This claim is made clearer by the next example.

\begin{ex}
\label{ex:deletion}
We have already noted that the sequence$$S(0,1,4)=0,1,4,5,11,12,14,15,31,32,34,35,40,41,43,44,89,\ldots$$ is dependent, with core $S(0)=0,1,3,4,9,\ldots$.  Removing $11$ from $S(0,1,4)$ yields the sequence $$0,1,4,5,12,14,15,31,32,34,35,40,41,43,44,89,\ldots,$$ which may be expressed as $S(0,1,4,5,12,14,15,31)$.  Furthermore, it is evident that this Stanley sequence is dependent, with core $S(0)$ and shift index $\s=1$, since one element was removed.

Likewise, removing both $31$ and $32$ from $S(0,1,4)$ yields the dependent Stanley sequence $$S(0,1,4,5,11,12,14,15,34,35,40,41,43,44,89),$$which has core $S(0)$ and shift index $\s=2$, since two elements were removed.
\end{ex}

\begin{rem}
It follows from Theorem \ref{thm:product} that the shift index $\s$ can be arbitrarily large. If $A=\{0\}$, $B=\{0,1,4,5,12,14,15,31\}$, and $k$ is sufficiently large, the sequence $S(A\otimes_k B)$ is a dependent sequence satisfying $$\s(A\otimes_k B)=2^k\cdot \s(B)=2^k.$$
\end{rem}

Given a dependent Stanley sequence $S(A)$, we say that an element of $S(A)$ is \emph{deletable} if deleting it yields another (dependent) Stanley sequence.  We have as yet been unable to derive a general formulation for which elements of a Stanley sequence are deletable.

\begin{conj} Every dependent Stanley sequence contains infinitely many deletable elements.
\end{conj}

\section{Concluding remarks}

In this paper, we have rigorously identified and explored the notion of regularity in Stanley sequences and have constructed many new classes of regular sequences.  Our research suggests several avenues for further exploration.  Most significant, perhaps, among these is the problem of whether all irregular sequences satisfy Type 2 growth.  A better upper bound on the asymptotic density of irregular sequences would be welcome.  Roth's theorem \cite{roth} implies that $a_n$ cannot grow linearly with $n$.  A result by Sanders \cite{sanders} strengthens this bound slightly to $n\log^{1-o(1)} n$. However, no explicit bound of the form $\Omega(n^{1+\epsilon})$ has been found (see Problem 2 of Erd\H{o}s et al.~\cite{erdos}).  We hypothesize that restating the problem using Fourier analysis could shed light on this and other questions in Stanley sequence theory.  As the problem is currently stated, complete classification of the regular Stanley sequences seems to us impracticable, but this may become easier when Stanley sequences are set up in Fourier analytical terms.

We conclude by offering an intriguing conjecture of a different flavor.  Given a 3-free set $A$ with elements $a_0<a_1<\cdots<a_k$, define a \emph{completion} of $A$ to be a 3-free set $A'$ with elements $a_0<a_1<\cdots<a_k<\cdots <a_m$ such that $S(A')$ is regular.  For instance, $\{0,4,7\}$ and $\{0,4,9\}$ are two different completions of $\{0,4\}$.

\begin{conj} Every 3-free set has a completion.
\end{conj}

\section{Acknowledgments}
This research was performed at the University of Minnesota Duluth REU and was supported by the National Science Foundation (grant number DMS-1062709) and the National Security Agency (grant number H98230-11-1-0224).

I would like to thank Joe Gallian for supervising this research, Ricky Liu for offering many insightful questions and ideas, and all the other students and participants in the Duluth REU for their advice and support.

\bibliography{stanley.bib}

\begin{thebibliography}{1}

\bibitem{erdos}
P.~Erd\H{o}s, V.~Lev, G.~Rauzy, C.~S\'{a}ndor, and A.~S\'{a}rk\"{o}szy.
\newblock Greedy algorithm, arithmetic progressions, subset sums and
  divisibility.
\newblock {\em Discrete Math.}, 200:119--135, 1999.

\bibitem{lindhurst}
S.~Lindhurst.
\newblock An investigation of several interesting sets of numbers generated by
  the greedy algorithm, 1990.
\newblock Senior thesis, Princeton University.

\bibitem{moy}
R.~A. Moy.
\newblock On the growth of the counting function of stanley sequences.
\newblock {\em Discrete Math.}, 311:560--562, 2011.

\bibitem{stanley}
A.~M. Odlyzko and R.~P. Stanley.
\newblock Some curious sequences constructed with the greedy algorithm, 1978.
\newblock Bell Laboratories internal memorandum.

\bibitem{roth}
K.~F. Roth.
\newblock On certain sets of integers.
\newblock {\em J.~London Math.~Soc.}, 28:104--109, 1953.

\bibitem{sanders}
T.~Sanders.
\newblock On roth's theorem on progressions.
\newblock {\em Ann.~of Math.}, 174:619--636, 2011.

\end{thebibliography}
\bibliographystyle{plain}

\section*{Appendix: Sequences with small character}
\label{char}
We have found independent Stanley sequences $S(A)$ for each possible character $\lb$ such that $0\le \lb \le 76$, with the exception of the values $1,3,5,9,11,15$. The following table gives examples.
$$\begin{array}{c|c||c|c||c|c}
\lb&A&\lb&A&\lb&A \\ \hline
0&\{0\}&26&\{0, 5, 9, 12\}&52&\{0, 1, 10, 13, 14, 23\}\\
1&\text{None}&27&\{0, 10, 11, 17\}&53&\{0, 23, 24, 30\}\\
2&\{0, 2\}&28&\{0, 3, 4, 7, 22, 25\}&54&\{0, 4, 16, 21, 25\}\\
3&\text{None}&29&\{0, 3, 5, 8, 21, 24, 26\}&55&\{0, 3, 28\}\\
4&\{0, 2, 5\}&30&\{0, 6, 15\}&56&\{0, 5, 9, 17, 24\}\\
5&\text{None found}&31&\{0, 5, 11, 13, 16\}&57&\{0, 3, 19, 22, 29\}\\
6&\{0, 3\}&32&\{0, 6, 8, 15\}&58&\{0, 1, 3, 4, 29\}\\
7&\{0, 1, 7\}&33&\{0, 3, 7, 10, 21, 24, 28, 30\}&59&\{0, 3, 21, 30\}\\
8&\{0, 3, 5\}&34&\{0, 8, 17\}&60&\{0, 7, 19, 27\}\\
9&\text{None found}&35&\{0, 9, 10, 13, 19, 22\}&61&\{0, 5, 13, 18, 24, 28\}\\
10&\{0, 1, 4, 6, 10\}&36&\{0, 18\}&62&\{0, 8, 12, 20, 27\}\\
11&\text{None found}&37&\{0, 1, 19\}&63&\{0, 4, 9, 13, 30, 33\}\\
12&\{0, 6\}&38&\{0, 3, 11, 18\}&64&\{0, 3, 9, 12, 31, 34\}\\
13&\{0, 2, 7, 9, 13\}&39&\{0, 11, 15, 16, 20, 26, 28\}&65&\{0, 5, 17, 22, 28, 30, 33\}\\
14&\{0, 3, 8\}&40&\{0, 2, 7, 15, 16, 20\}&66&\{0, 10, 22, 27, 30\}\\
15&\text{None found}&41&\{0, 3, 11, 14, 21, 24, 30\}&67&\{0, 11, 23, 24, 28, 34\}\\
16&\{0, 4, 7\}&42&\{0, 9, 12, 13, 21\}&68&\{0, 3, 11, 12, 23, 27, 30\}\\
17&\{0, 4, 5, 9, 15, 17\}&43&\{0, 1, 9, 10, 25\}&69&\{0, 3, 4, 19, 22, 23, 28\}\\
18&\{0, 9\}&44&\{0, 14, 18, 21\}&70&\{0, 7, 9, 19, 27, 34\}\\
19&\{0, 3, 10\}&45&\{0, 1, 16, 17, 19, 20, 29\}&71&\{0, 8, 9, 17, 30, 33, 38\}\\
20&\{0, 1, 10\}&46&\{0, 1, 4, 12, 19\}&72&\{0, 13, 25, 27, 33\}\\
21&\{0, 1, 3, 4, 21\}&47&\{0, 20, 21, 27\}&73&\{0, 4, 5, 9, 15, 17, 20, 28\}\\
22&\{0, 8, 9, 14\}&48&\{0, 1, 12, 13, 21\}&74&\{0, 14, 17, 26, 27, 33\}\\
23&\{0, 7, 9, 10, 16\}&49&\{0, 9, 25\}&75&\{0, 4, 13, 17, 25, 29, 38\}\\
24&\{0, 9, 12\}&50&\{0, 2, 12, 14, 21\}&76&\{0, 1, 7, 8, 21, 28\}\\
25&\{0, 2, 3, 5, 23, 25\}&51&\{0, 5, 13, 16, 18, 24, 28\}&
\end{array}$$
\end{document}